\documentclass[english,a4paper,10pt]{article}

\usepackage{fullpage}

\usepackage[T2A]{fontenc}
\usepackage[utf8]{inputenc}
\usepackage{hhline}
\usepackage{cite,amsmath,amssymb}
\usepackage{calc}
\usepackage[pdftex]{graphicx}
\usepackage{amssymb,latexsym,amsmath,epsfig,amsthm,tabularx,chngpage} 

\usepackage{hyperref}
\usepackage{authblk}

\makeatletter
 
\renewcommand\section{\@startsection {section}{1}{\z@}
{-30pt \@plus -1ex \@minus -.2ex}
{2.3ex \@plus.2ex}
{\normalfont\normalsize\bfseries}}

\renewcommand\subsection{\@startsection{subsection}{2}{\z@}
{-3.25ex\@plus -1ex \@minus -.2ex}
{1.5ex \@plus .2ex}
{\normalfont\normalsize\bfseries}}

\renewcommand{\@seccntformat}[1]{\csname the#1\endcsname. }

\makeatother

\newtheorem{theorem}{Theorem}

\newtheorem{corollary}{Corollary}

\newcommand{\beq}{\begin{equation}}
\newcommand{\eeq}{\end{equation}}

\def\({\left(}
\def\){\right)}
\def\[{\left[}
\def\]{\right]}

\title{On natural  densities of sets of some type integers}
\author{Dmitry I. Khomovsky}
\affil{\small Lomonosov Moscow State University, Moscow, Russia. Email: \href{mailto:khomovskij@physics.msu.ru}{khomovskij@physics.msu.ru}}

\begin{document}

\maketitle

\begin{abstract}
Let $a_0=b_0=0$ and $0<a_1\leq b_1<a_2\leq b_2<\ldots\leq b_{n}$ be integers.
Let $Q\left(x;\bigcup_{j=1}^{n}[a_j,b_j]\right)$ be the number of integers between $1$ and $x$ such that all exponents in their prime factorization are in $\bigcup_{j=1}^{n}[a_j,b_j]$. The following formula holds:
$$\lim_{x\to\infty}{\frac{Q\left(x;\bigcup_{j=1}^{n}[a_j,b_j]\right)}{x}}=\prod\limits_{p}\sum\limits_{i=0}^{n}\left(\frac{1}{p^{a_{i}}}-\frac{1}{p^{b_{i}+1}}\right).$$
In this paper, we prove this result and then generalize it.
\end{abstract}

\section{Introduction}
In  \cite{R}, Rényi and Turán proved the following theorem.
\begin{theorem}
Let $\Delta$ be a function on natural numbers such that $\Delta(1)=0$ and for a positive integer $u =\prod_{i=1}^r {p_i}^{\alpha_i}>1$ the following holds: $\Delta(u)=\sum_{i=1}^r(\alpha_i-1).$
Then, for any integer $k \geq 0$, there exists a density $d_k$ of the set of integers $y$ such that $\Delta(y)=k$, and the following identity holds:
\beq
\sum\limits_{k=0}^{\infty}d_k z^k=\prod\limits_{p}\left(1-\frac{1}{p}\right)\left(1+\frac{1}{p-z}\right),
\eeq
where $|z|<2.$
\end{theorem}
It should be noted that if we put $z = 0$, then we get the density of square-free numbers.
We will show that other densities of subsets of integers can be obtained by analogy. To do this, we need to modify the function $\Delta$ in a certain way and carry out the calculations that were used in the paper \cite{R}.

We need the following classical result \cite{Ap}.
\begin{theorem}
Let  $f$ be a multiplicative arithmetic function such that the series $\sum_{n=1}^{\infty}f(n)$ is absolutely convergent, then 
\beq\sum_{n=1}^{\infty}f(n)=\prod\limits_{p}\left(1+\sum\limits_{i=1}^{\infty}f(p^i)\right).\eeq
\end{theorem}
\begin{corollary}
Let $g$ be a function on natural numbers such that $g(1)=0$ and for a positive integer $u =\prod_{i=1}^r {p_i}^{\alpha_i}>1$ the following holds: $g(u)=\sum_{i=1}^r h(\alpha_i)$, where $h\colon \mathbb{Z}_+\to \{0,1\}$ is an arbitrary function.  Let $z$ and $s$ be complex numbers such that the Dirichlet series $\sum_{n=1}^{\infty}{z^{g(n)}}/{n^s}$ is absolutely convergent, then
\beq\sum\limits_{n=1}^{\infty}\frac{z^{g(n)}}{n^s}=\prod_{p}\left(1+\sum_{i=1}^{\infty}\frac{ z^{h(i)}}{(p^s)^i}\right).\eeq
\end{corollary}

\section{Main theorems}
The following result was obtained and proved by Shevelev, see \cite{VS,S,S1}. 
Note also that it is easy to obtain using the theorem $1$ presented by Tóth in \cite{TT}.
We give one more proof.
\begin{theorem}
Let $a_0=b_0=0$ and $0<a_1\leq b_1<a_2\leq b_2<\ldots\leq b_{n}$ be integers.
Let $Q\left(x;\bigcup_{j=1}^{n}[a_j,b_j]\right)$ be the number of integers between $1$ and $x$ such that all exponents in their prime factorization are in $\bigcup_{j=1}^{n}[a_j,b_j]$. The following formula holds:
\beq\label{T}\lim_{x\to\infty}{\frac{Q\left(x;\bigcup_{j=1}^{n}[a_j,b_j]\right)}{x}}=\prod\limits_{p}\sum\limits_{i=0}^{n}\left(\frac{1}{p^{a_{i}}}-\frac{1}{p^{b_{i}+1}}\right).\eeq
\end{theorem}
\begin{proof}
Let $g$ be a function that satisfies all requirements of  Corollary $1$. Let the function $h$ be such that if $\alpha\in\bigcup_{j = 1}^{n}[a_j, b_j]$, then $h(\alpha) = 0$, else $h(\alpha)= 1.$ Consider the Dirichlet series
\beq\label{T1}
\delta(s,z)=\sum\limits_{n=1}^{\infty}\frac{z^{g(n)}}{n^s}.
\eeq
Using Corollary $1$ we get
\beq\label{T2}
\delta(s,z)=\zeta(s)\prod_{p}\left(1-\frac{1}{p}\right)\left(1+\sum_{i\in I}\frac{1}{(p^s)^i}+\sum_{i\notin I}\frac{z}{(p^s)^i}\right),
\eeq
where $I=\bigcup_{j=1}^{n}[a_j,b_j]$.
Using the reasoning presented in \cite{R} and performing the necessary calculations, we get
\beq
\lim\limits_{m\to\infty}\frac{1}{m}\sum\limits_{n=1}^{m}z^{g(n)}=\sum\limits_{k=0}^{\infty}d_k z^k=\prod_{p}\left(1-\frac{1}{p}\right)\left(1+\sum_{i\in I}\frac{1}{p^i}+\sum_{i\notin I}\frac{z}{p^i}\right), \text{ where } |z|\leq 1.
\eeq
Here, $d_k$ is the natural density of the subset of integers $y$ such that $g(y)=k.$ Thus
\beq\label{T3}d_0=\prod_{p}\left(1-\frac{1}{p}\right)\left(1+\sum_{i\in I}\frac{1}{p^i}\right)=\prod\limits_{p}\sum\limits_{i=0}^{n}\left(\frac{1}{p^{a_{i}}}-\frac{1}{p^{b_{i}+1}}\right), \,\,\, a_0=b_0=0.\eeq
It is clear that
\beq\lim_{x\to\infty}{\frac{Q\left(x;\bigcup_{j=1}^{n}[a_j,b_j]\right)}{x}}=d_0.\eeq
\end{proof}
Below we demonstrate examples of using Theorem $3$. Some of them are well-known.
\begin{align} 
\lim_{x\to\infty}{\frac{Q\left(x;[1,k]\right)}{x}}=&\prod_{p}\left(1-\frac{1}{p^{k+1}}\right)=\frac{1}{\zeta(k+1)}, \\ 
\lim_{x\to\infty}{\frac{Q\left(x;[1,1]\cup[k,\infty]\right)}{x}}=&\prod_{p}\left(1-\frac{1}{p^2}+\frac{1}{p^k}\right),\\
\lim_{x\to \infty }{\frac{Q(x;[1,k-1]\cup[k+1,\infty])}{x}}=&\prod\limits_p\left(1-\frac{1}{p^k}+\frac{1}{p^{k+1}}\right),\\
\label{f}
\lim_{x\to \infty }{\frac{Q(x;\bigcup_{j=0}^{\infty}[2j+1,2j+1])}{x}}=&\prod\limits_p\left(1-\frac{1}{p (p+1)}\right),
\\
\lim_{x\to \infty }{\frac{Q(x;\bigcup_{j=0}^{\infty}[\ell j+1,\ell j+1])}{x}}=&\prod\limits_p\left(1-\frac{p^{\ell-1}-1}{p (p^{\ell}-1)}\right),\,\,\, \ell\in\mathbb{Z}_+.
\end{align}
The formula $(\ref{f})$ gives the density of exponentially odd numbers, see the sequence $A065463$ in Sloane’s Online Encyclopedia of Integer Sequences \cite{OEIS}.
Other interesting results can be found in \cite{T0,T,T1,D,N,Dr}.

Let us consider the generalized theorem.
\begin{theorem}
Let $g$ be a function on natural numbers such that $g(1)=0$ and for a positive integer $u =\prod_{i=1}^r {p_i}^{\alpha_i}>1$ the following holds: $g(u)=\sum_{i=1}^r h(\alpha_i,p_i)$, where $h\colon \mathbb{Z}_+\times\mathbb{P}\to \{0,1\}$ is an arbitrary function.
Let $Q\left(x;h\right)$ be the number of integers $n$ between $1$ and $x$ such that $g(n)=0$. The following formula holds:
\beq\label{PP}\lim_{x\to\infty}{\frac{Q(x;h)}{x}}=\prod\limits_{p}\left(1-\frac{p-1}{p}\sum\limits_{i=1}^{\infty}\frac{h(i,p)}{p^i}\right).\eeq
\end{theorem}
\begin{proof}
Similarly to the proof of the previous theorem, we have
\beq\label{P1}
\delta(s,z)=\zeta(s)\prod_{p}\left(1-\frac{1}{p}\right)\left(1+\sum_{i=1}^{\infty}\frac{z^{h(i,p)}}{(p^s)^i}\right).
\eeq
Since $\lim\limits_{z\to0}z^{h(i,p)}=1-h(i,p)$, we get
\beq\label{P2}d_0=\prod\limits_{p}\left(1-\frac{p-1}{p}\sum\limits_{i=1}^{\infty}\frac{h(i,p)}{p^i}\right).\eeq
Finally we conclude that $\lim\limits_{x\to\infty}{\frac{Q(x;h)}{x}}=d_0.$
\end{proof}
This theorem is a powerful tool for obtaining  densities of various subsets of integers. Let us show it.

\noindent{\bf Example $1$.} Suppose that we need to obtain the natural density of numbers $n=\prod_{i=1}^r {p_i}^{\alpha_i}$ such that $p_i>q$ and $\alpha_i<k$ for all $1\leq i\leq r,$ here $q$ is some prime. In other words, these are $k$th-power-free numbers for which $\gcd(n,q\#)=1$. We need to use the following function:
\beq
h(\alpha,p)=\begin{cases} 1 &\mbox{if } p \leq q \\
1 & \mbox{if } p>q \text{ and } \alpha \geq k \\
0 & \mbox{if } p>q \text{ and } \alpha < k.
\end{cases}
\eeq
By using $(\ref{PP})$ we get
\beq\lim_{x\to\infty}{\frac{Q(x;h)}{x}}=\frac{\prod\limits_{p\leq q}\left(1-\frac{1}{p}\right)}{\zeta(k)\prod\limits_{p\leq q}\left(1-\frac{1}{p^k}\right)}.\eeq

\noindent{\bf Example $2$.} 
Let $S =\{q_1,q_2,\ldots,q_m\} $ be the set of arbitrary distinct primes. Consider $k$th-power-free numbers $n$ such that $\gcd(n,q_i) = 1$ for all $1\leq i\leq m$. To find the density we have to use 
\beq
h(\alpha,p)=\begin{cases} 1 &\mbox{if } p\in S \\
1 & \mbox{if } \alpha \geq k \\
0 & \mbox{if } p\notin S \text{ and } \alpha < k.
\end{cases}
\eeq
Then
\beq\lim_{x\to\infty}{\frac{Q(x;h)}{x}}=\frac{1}{\zeta(k)}\prod\limits_{q\in S}\frac{q^k-q^{k-1}}{q^k-1}.\eeq

\noindent{\bf Example $3$.} 
Let us consider the numbers $n=\prod_{i=1}^r {p_i}^{\alpha_i}$ for which one of exponents $\alpha_i$ can be any number, and the rest are less than $k$. We denote the density of such numbers by $\rho$.
First, we need to obtain a density of numbers for which the prime factorization can have any exponent for prime $p$, and the other prime exponents less than $k$.
With the help of $(\ref{PP})$ we conclude that this density is
\beq\rho_p=\frac{1}{\zeta(k)}\frac{p^k}{p^k-1}.\eeq
Finally we have
\beq\rho=\frac{1}{\zeta(k)}+\sum\limits_{p\in \mathbb{P}}\left(\rho_p-\frac{1}{\zeta(k)}\right)=\frac{1}{\zeta(k)}\left(1+\sum\limits_{p\in \mathbb{P}}\frac{1}{p^k-1}\right).\eeq

\smallskip
\noindent{\bf Acknowledgments.}
The author is grateful to the participants of the internet forum {\tt https://dxdy.ru} for their fruitful discussions.

\end{document}